\documentclass[11pt,a4,leqno]{amsart} 
\usepackage[top=25mm, bottom=25mm, left=25mm, right=25mm]{geometry}
\usepackage{amsthm}
\usepackage{amsmath}
\usepackage{amssymb}
\usepackage{amsfonts}
\usepackage{mathrsfs}
\usepackage{graphicx}
\usepackage{bbm}
\usepackage[
colorlinks=true, citecolor=blue, linkcolor=blue, urlcolor=red]{hyperref}
\usepackage{color}
\usepackage[all]{xy}
\usepackage{comment}
\usepackage{here}
\usepackage{amscd}
 \usepackage{enumitem}
\allowdisplaybreaks

\makeatletter

\@addtoreset{equation}{section}
\makeatother

\makeatletter
\@namedef{subjclassname@2020}{%
  \textup{2020} Mathematics Subject Classification}
\makeatother

\newtheoremstyle{my theoremstyle}
{1.0em}                    
    {1.0em}                    
    {\itshape}                   
    {}                         
    {\scshape}                   
    {.}                          
    {.5em}                      
    {} 

\newtheoremstyle{dfn}
{1.0em}                    
    {1.0em}                   
    {}                 
    {}                        
    {\scshape}                 
    {.}                         
    {.5em}                       
    {}  
    
\theoremstyle{my theoremstyle}
   \newtheorem{thm}{Theorem}[section]
   
   \newtheorem{prop}[thm]{Proposition}
   \newtheorem{cor}[thm]{Corollary}
     \newtheorem{conj}[thm]{Conjecture}
     
\theoremstyle{dfn}
   \newtheorem{dfn}[thm]{Definition}
   
\theoremstyle{remark}   
   \newtheorem{exa}[thm]{{\scshape Example}}

\newcommand{\F}{\mathbb{F}}
\newcommand{\Z}{\mathbb{Z}}
\newcommand{\Q}{\mathbb{Q}}
\newcommand{\bQ}{\overline{\mathbb{Q}}}
\newcommand{\R}{\mathbb{R}}
\newcommand{\C}{\mathbb{C}}

\newcommand{\p}{\mathbb{P}} 
\newcommand{\ord}{\operatorname{ord}}
\newcommand{\jac}{\operatorname{Jac}}

\newcommand{\al}{\alpha}
\newcommand{\be}{\beta}
\newcommand{\D}{\mathscr{D}}
\newcommand{\la}{\lambda}
\newcommand{\lra}{\longrightarrow}
\newcommand{\ds}{\displaystyle}

\newcommand{\spec}{\operatorname{Spec}}
\newcommand{\Hom}{\operatorname{Hom}}
\numberwithin{equation}{section}

\date{\today}

\begin{document}

\title{Regulators on some abelian coverings of $\mathbb{P}^1$ minus $n+2$ points}
\author{Yusuke Nemoto}
\address{Graduate School of Science, Chiba
University, Yayoicho 1-33, Inage, Chiba, 263-8522, Japan.}
\email{y-nemoto@waseda.jp}

\author{Takuya Yamauchi}
\address{Takuya Yamauchi \\
Mathematical Inst. Tohoku Univ.\\
 6-3,Aoba, Aramaki, Aoba-Ku, Sendai 980-8578, JAPAN}
\email{takuya.yamauchi.c3@tohoku.ac.jp}

\date{\today}

\maketitle

\begin{abstract}
In this paper, we construct certain rational or integral elements in the motivic cohomology of superelliptic curves which are quotient curves of abelian coverings of $\mathbb{P}^1$ minus $n+2$ points, and prove that these elements are non-trivial by expressing their regulators in terms of Appell-Lauricella hypergeometric functions. We also check that such elements are integral under 
a mild assumption. 
We also give various numerical examples for the Beilinson conjecture on special values of $L$-functions 
of the superelliptic curves by using hypergeometric expressions. 
\end{abstract}

\tableofcontents

\section{Introduction} 
The Beilinson conjecture is a vast generalization of the class number formula, which asserts that the special values of the 
$L$-functions for algebraic varieties over number fields are explained by the regulator map (cf. \cite{Beilinson},\cite{RSS},\cite{Nekovar}).  
In \cite[Lecture 8,9]{Bloch},  Bloch has originally studied it for elliptic curves over $\Q$ with special emphasis on CM 
elliptic curves. Thus, the conjecture is sometimes called the Bloch-Beilinson conjecture. 

For a smooth projective geometrically irreducible algebraic curve $C$ of genus $g$ over a number field $k$, the Beilinson conjecture claims the regulator map  
\begin{equation}\label{rD}
r_{\D} \colon H^2_{\mathscr{M}}(C, \Q(2))_{\Z} \lra H^2_{\D} (C_{\R}, \R(2)) \simeq \Hom_\R (H_1(C(\C), \Z)^-, \R)
\end{equation}
leads to a $\Q$-structure of the target and a non-degenerate pairing $$\langle\ast,\ast\rangle:H^2_{\mathscr{M}}(C, \Q(2))_{\Z}\times 
 H_1(C(\C), \Z)^-\lra \R,\ (\alpha,\gamma)\mapsto r_{\D}(\alpha)(\gamma).$$
Here, the left-hand side of (\ref{rD}) is the integral part of 
the motivic cohomology $H^2_{\mathscr{M}}(C, \Q(2))$, and the right-hand side is the real Deligne cohomology (see Section \ref{Beilinson conjecture} for notation and definition). 
Further, it is conjectured that the discriminant of the pairing $\langle\ast,\ast\rangle$ for a suitable basis is 
proportional over $\Q^\times$ to $\ds L(C,2)\pi^{-2g[k:\Q]}$ where $L(C,s)$ stands for the $L$-function of $C$. 
If $L(C,s)$ has been analytically extended to the whole $\C$ as a holomorphic function in $s$, we can replace  $\ds L(C,2)\pi^{-2g[k:\Q]}$ 
with $L^{(g:[k:\Q])}(C,0)$ to compare it with the discriminant where $L^{(r)}(C,s)$ stands for the $r$-th derivative of $L(C,s)$ 
(see \cite[Conjecture 2.1]{DJZ}). 
After Bloch and Beilinson, many people have attacked this conjecture both theoretically (cf. \cite{Be86},\cite{Bru07}) and 
numerically (cf. \cite{DJZ}, \cite{LJ15}). 
In particular, there are many works about the regulators of algebraic curves (cf. \cite{Ross1, Ross2, Asakura, Otsubo, Tokiwa, Nemoto}). For example, let $F_N$ be the Fermat curve of degree $N$ over $\Q$ whose affine smooth model is defined by 
$$x^N + y^N=1.$$
Ross \cite{Ross1} proves  the Milnor $K_2$-symbol $\{1-xy, x\}^{6N}$ belongs to $H^2_{\mathscr{M}}(F_N, \Q(2))_{\Z}$ and proves that 
it is non-trivial by expressing the regulator in terms of an infinite sum of the Beta functions $B(s, t)$: 
there exists a $1$-cycle $\delta \in H_1(F_N(\C), \Z)$ such that 
$$r_{\D}(\{1-xy, x\}^{6N})(\delta)=-\dfrac{3}{2 \pi} \sum_{k=1}^{\infty} \frac1k \sin \frac{2 \pi k}N \left(1- \cos \frac{2 \pi k}{N} \right) B\left(\frac{k}N, \frac{k}N \right) \in \R^\times. $$
Otsubo \cite{Otsubo} generalizes Ross's result to Fermat motives by expressing the regulators in terms of the special value of hypergeometric functions ${}_3F_2(1)$. 
For the relationship  between regulators of algebraic curves and hypergeometric functions, see also \cite{Asakura, Tokiwa, Nemoto}. 

In this paper, we generalize the above result to the quotient curve $E_{N,n}$ (defined as below) of the generalized Fermat curve $X_{N, n}=X_{N, n}(\la_1, \ldots, \la_n)$ 
over $k_{X_{N,n}}:=\Q(\la_1, \ldots, \la_n)$ whose affine smooth model inside $\mathbb{A}^{n+1}$ with coordinates $(x,y_1,\ldots,y_n)$ is given by 
\begin{align*}
X_{N, n} : \left\{
\begin{array}{l}
x^N+(-1)^ny_1^N=\lambda_1 \\
x^N+(-1)^ny_2^N=\lambda_2 \\
\qquad \vdots \\
x^N+(-1)^ny_n^N=\lambda_n, 
\end{array}
\right.
\end{align*}
where $\lambda_1, \ldots, \lambda_n \in \bQ^\times$ are different from each other. 
We consider the Milnor $K_2$-symbol 
$$\xi_{X_{N, n}}:= \{1-xy_1 \cdots y_n, x\} \in K_2^M(k_{X_{N,n}}(X_{N, n})).$$
 
Let $k_{E_{N,n}}=\Q(\la_1,\ldots,\la_n)^{S_n}$ where the $n$-th symmetric group $S_n$ acts on $\Q(\la_1,\ldots,\la_n)$ by permuting $\la_1,\ldots,\la_n$.

Our target is the superelliptic curve $E_{N, n}=E_{N, n}(\la_1, \ldots, \la_n)$ defined over 
$k_{E_{N,n}}$ whose affine smooth 
model is given by 
$$Y^N=X(X-\lambda_1) \cdots (X-\lambda_n).$$
It is a quotient curve of $X_{N, n} $ endowed with the quotient morphism given in terms of affine coordinates:
$$p \colon X_{N, n} \to E_{N, n}; \quad (x, y_1, \ldots, y_n) \mapsto (X, Y)=(x^N, xy_1 \ldots y_n). $$
It is easy to see that the Milnor $K_2$-symbol $\{1-Y, X\}$ of $E_{N, n}$ satisfies 
$p^*(\{1-Y, X\})=N \xi_{X_{N, n}}$ and $p^*$ is injective by the norm argument (cf. \cite[p.226]{Ross2}). 

We first show the following result:
\begin{thm}\label{main1}{\rm(}Theorem \ref{mot}, Proposition \ref{intE}, Corollary \ref{intX}{\rm)} Keep the notation as above. Assume all roots of the polynomial 
$$\Phi_{X_{N,n}}(T):=T\Big(\ds\prod_{i=1}^n (T-\la_i)\Big)-1$$ are roots of unity. 
Then, $\xi_{X_{N, n}}$ {\rm(}resp. $\{1-Y, X\}${\rm)} is an element in the motivic rational cohomology $H^2_{\mathscr{M}} (X_{N, n}, \Q(2))\subset  K_2^M(k_{X_{N,n}}(X_{N, n}))$ (resp. $H^2_{\mathscr{M}} (E_{N, n}, \Q(2))\subset  K_2^M(k_{E_{N,n}}(E_{N, n}))$.  

Further, assume $\Phi_{X_{N,n}}(T)\in \Z[T]$. Then,  
$\xi_{X_{N, n}}$ and $\{1-Y, X\}$ are both integral, namely, they are elements of 
$H^2_{\mathscr{M}} (X_{N, n}, \Q(2))_\Z$ and $H^2_{\mathscr{M}} (E_{N, n}, \Q(2))_\Z
\subset K_2^M(\Q(E_{N, n}))$ 
respectively.   
\end{thm}

For above $\Phi_{X_{N,n}}(T)$, put $f(T)=\Phi_{X_{N,n}}(T)+1$.
Suppose that 
\begin{equation}\label{thmassump}
\left\{\begin{array}{l}
\text{(i) all roots of 
the polynomial $\Phi_{X_{N,n}}$ are roots of unity,} \\
\text{(ii) All coefficients of $f(T)$ are real numbers,} \\
\text{(iii) $|f(t)|<1$ for $0\le t<1$ and $f(1)>0$,} \\
\text{(iv) there exists a unique $i$ ($1\le i \le n$) such that $0<\la_i<1$,} \\
\text{(v) $(-1)^{n-1}\la_1\cdots\la_n>0$}.
\end{array}\right.
\end{equation}
These conditions look technical but we can easily find infinitely many examples 
by using factors of $T^m-1$ where $m$ varies (cf. Example \ref{ex1}, \ref{ex2}). 
We define $$\gamma \colon [0,1]\lra E_{N,n}(\C),\ t\mapsto [t,\sqrt[N]{f(t)}]$$ where 
we take $\sqrt[N]{f(t)}$ to mean the principal branch of the $N$-th root of $f(t)\in \R$ for $t\in [0,1]$. 
It gives rise to an element 
$$\delta=\gamma-\overline{\gamma}\in H_1(E_{N,n}(\C),\Z)^{-}.$$
 We are now ready to explain our main result:
\begin{thm}[Theorem \ref{main:1}]  \label{main:intro}
Under the assumption (\ref{thmassump}), it holds that 
\begin{enumerate}[label=\textnormal{(\arabic*)}]
\item 
$r_{\D}(\{1-Y, X\})(\delta) =-\ds\frac{1}{\pi} \notag
\sum_{k=1}^{\infty}
\frac1k
 ((-1)^{n-1}\la_1 \cdots \la_n)^{\frac{k}N}\la_n^{\frac{k}{N}}
\sin\frac{\pi k}{N}B\left(\frac{k}{N}, \frac{k}{N} \right) $
$$\qquad  \qquad \quad \quad  \times F^{(n-1)}_D \left. \left( 
\frac{k}{N}, -\frac{k}{N}, \ldots, -\frac{k}{N}, \frac{2k}{N}+1
\ \right| 
\frac{\la_n}{\la_1}, \ldots, \frac{\la_{n}}{\la_{n-1}}
\right) 
$$
\item the right-hand side of the above expression is non-zero. 
\end{enumerate}
Here, $F^{(n)}_D(\ |\ x_1, \ldots, x_n)$ is an Appell-Lauricella hypergeometric function, which is a generalization of a Gaussian hypergeometric function (see Section \ref{Appell_HG}). 

In particular, the rational elements $\{1-Y,X\}\in H^2_{\mathscr{M}}(E_{N, n}, \Q(2))
\subset K_2^M(k_{E_{N,n}}(E_{N, n}))$ and 
$\{1-xy_1 \cdots y_n, x\} \in H^2_{\mathscr{M}}(X_{N, n}, \Q(2))\subset K_2^M(k_{X_{N,n}}(X_{N, n}))$ are both non-trivial. 
\end{thm}

As an application of Theorem \ref{main:intro}, in Section \ref{numerical_Beilinson}, we numerically test the Beilinson conjecture 
for certain $E_{N,n}$ when $E_{N,n}$ has small genus.  

Our construction of motivic elements is a kind of covering constructions so that a covering $X\lra Y$ with a well-understood 
$X$ for motivic elements yields motivic elements for $Y$ from $X$. For example, the case when $Y=E$ is an elliptic curve over $\Q$ with 
a modular parametrization $X=X_0(N)\lra E$ is nothing but the case Beilinson \cite{Be86} handled. 
A key of our construction is to use the covering $X_{N,n}$ of $E_{N,n}$ to figure out easily verifiable conditions on $(\la_1,\cdots,\la_n)$ 
to find a motivic element. It may be hard to directly work on the general superelliptic curves.

This paper is organized as follows. In Section \ref{Preparation}, we recall some properties of the Appell-Lauricella hypergeometric functions and the statement of the Beilinson conjecture for curves. 
In Section \ref{elements}, we construct an element in the motivic cohomology of the generalized Fermat curve under some assumptions. We also prove that this element defines an element of the integral part of the motivic cohomology (Proposition \ref{intE}).  
In Section \ref{computation}, we compute the regulator of the element constructed in Section \ref{elements}, and express it in terms of Appell-Lauricella hypergeometric functions. 
In Section \ref{numerical_Beilinson}, we numerically test the Beilinson conjecture over $\Q$ and $\Q(\sqrt{-3})$ when 
$E_{N,n}$ has small genus and an extra automorphism.

\subsection*{Acknowledgments} 
This work grew out of Yuta Tokiwa's master's thesis on the covering construction for $n=2$ which was suggested by 
the second author. We would like to thank him for allowing us to extend his work. 
We are also grateful to Professor Noriyuki Otsubo for many helpful discussions. 
We thank Chiba University and Tohoku University for their incredible hospitality during the course of this research. 
The first author was supported by
Waseda University Grant for Early Career Researchers (Project number: 2025E-041).

\subsection*{Notations} 
For a positive integer $N$, $\mu_N \subset \overline{\Q}^{\times}$ denotes the group of 
$N$-th roots of unity. 
We put $\zeta_N=\exp(2 \pi i/N)$.

\section{Preparation} \label{Preparation}

\subsection{Appell-Lauricella hypergeometric functions} \label{Appell_HG}
In this subsection, we recall Appell-Lauricella hypergeometric functions. We refer to \cite{Slater} for basics.  
\begin{dfn}[cf. \cite{Slater}]
Appell-Lauricella's function $F_D^{(n)}$ is defined by
\begin{align*}
&F^{(n)}_D \left. \left( 
a, b_1, \ldots, b_n, c
\ \right| 
x_1, \ldots, x_n
\right) 
=
\sum_{m_1=0}^{\infty} \cdots \sum_{m_n=0}^{\infty}
\frac{(a)_{m_1+\cdots + m_n}(b_1)_{m_1} \cdots (b_n)_{m_n}}{(c)_{m_1+\cdots + m_n}(1)_{m_1} \cdots (1)_{m_n}}x_1^{m_1}\cdots x_n^{m_n},
\end{align*}
where $a, b_i, c \in \C$ with $c \not \in \Z_{\leq 0}$ and $(a)_n=\Gamma(a+n)/\Gamma(a)$ is the Pochhammer symbol. 
\end{dfn}

When $n=1$, $F_D^{(1)}(x)$ corresponds to ${}_2F_1(x)$,  hence Appell-Lauricella functions are a generalization of Gaussian  hypergeometric functions. 
\begin{prop}[{\cite[Theorem 3.4.1]{KS}}]
If $0 < \operatorname{Re} (a) < \operatorname{Re} (c) $, then this function has the following integral representation:
\begin{align} \label{int}
B(a, c-a)
&F^{(n)}_D \left. \left( 
a, b_1, \ldots, b_n, c
\ \right| 
x_1, \ldots, x_n
\right) 
 \\
& = \notag
\int_0^1 \left(\prod_{i=1}^n (1-x_i u)^{-b_i}\right)
u^{a-1}
(1-u)^{c-a-1}du, 
\end{align}
where $B(s, t)$ is the Beta function.  
\end{prop}

\subsection{Beilinson conjecture} \label{Beilinson conjecture}
We recall the Beilinson conjecture for curves (cf. \cite{DJZ, Nekovar}).  
Let $k$ be a number field and $\mathscr{O}_k$ be the ring of integers of $k$. 
Let $C$ be a smooth projective geometrically irreducible curve of genus $g$ over $k$ and 
we denote by $k(C)$ the function field of $C$ over $k$.  We note that   
$C(\C)=\ds\coprod_{\sigma:k\hookrightarrow \C}C_\sigma(\C)$ 
where $C_\sigma=C\times_{k,\sigma}\C$. 
Let 
$$r_{\D} \colon H^2_{\mathscr{M}}(C, \Q(2)) \to H^2_{\D} (C_{\R}, \Q(2))$$
 be the regulator map defined by Beilinson \cite{Beilinson}. 
The left-hand side is the motivic cohomology group defined via $K$-theory.  
 The right-hand side is the Deligne cohomology group, for which we have an isomorphism (cf. \cite{Nekovar}) 
 $$H^2_{\D} (C_{\R}, \Q(2)) \simeq \operatorname{Hom}(H_1(C(\C), \Z)^-, \R). $$
Here, $-$ denotes the $(-1)$-eigenspace with respect to the \textit{de Rham conjugation} $F_{\infty} \otimes c_{\infty}$, where the \textit{infinite Frobenius} $F_{\infty}$ is the complex conjugation acting on $C(\C)$ and $c_{\infty}$ is the complex conjugation on the coefficients. 

We have an isomorphism
$$H^2_{\mathscr{M}}(C, \Q(2)) \cong \operatorname{Ker}\left(\tau \otimes \Q \colon K_2^M(k(C)) \otimes \Q \to \bigoplus_{x \in C^{(1)}}\kappa(x)^* \otimes \Q \right),$$
where $C^{(1)}$ is the set of all closed points on $C$, $\kappa(x)$ is the residue field, and $\tau=(\tau_x)$ is the tame symbol on $K_2^M(k(C))$ defined by
$$\tau_x\left( \{f, g\} \right)=(-1)^{\operatorname{ord}_x f \operatorname{ord}_x g}\left( \frac{f^{\operatorname{ord}_xg}}{g^{\operatorname{ord}_xf}}\right)(x). $$
The integral part of $H^2_{\mathscr{M}}(C, \Q(2)) $, denoted by $H^2_{\mathscr{M}}(C, \Q(2))_\Z$,  is defined by using 
any regular proper model of $C$ over $\mathscr{O}_k$ (see \cite[(3.5)]{DJZ} for the definition). 
The integral part is independent of the choice of such a model \cite[Section 3]{DJZ}.

For any embedding $\sigma \colon k \hookrightarrow \C$, the 
projection $C\times_{k,\sigma}k\lra C$ induces an isomorphism $k(C)\lra \sigma(k)(C)$.   
We write $f^{\sigma}\in \sigma(k)(C)$ as the image of $f$ under this isomorphism.   
The following proposition is a standard tool for explicit computation of the regulator:
\begin{prop}\label{regformula}{\rm(}\cite{Beilinson}, \cite[p.222, Proposition 4.4.4]{Ram}{\rm)}
Let $C$ be as above and $\{f, g\} \in H^2_{\mathscr{M}}(C, \Q(2))$. 
Fix an embedding $\sigma \colon k \hookrightarrow \C$.  
 Then, for any $\gamma \in H_1(C(\C), \Z)^-$, it holds   
\begin{align} \label{reg}
r_{\D}(\{f, g\})(\gamma)=\frac{1}{2 \pi} \operatorname{Im}\left( \int_{\widetilde{\gamma}} \log f^{\sigma} d \log g^{\sigma}- \log |g^{\sigma}(P_0)| \int_{\widetilde{\gamma}} d \log f^{\sigma} \right).
 \end{align}
Here, $\widetilde{\gamma}$ is a cycle representing $\gamma$ which aboids the zeros and poles of $f^{\sigma}$ and $g^{\sigma}$, and 
$P_0 \in C(\C)$ is a base point such that the path $\widetilde{\gamma}$ starts from $P_0$ to define 
the argument of $\log$. 
 \end{prop}

\begin{conj}[\cite{Beilinson}, \cite{DJZ, Nekovar}]Keep the notation as above. Then 
\begin{enumerate}[label=\textnormal{(\arabic*)}]
\item $H^2_{\mathscr{M}}(C, \Q(2))_{\Z}$ is a free abelian group of rank $g\cdot [k : \Q]$ and the pairing \eqref{reg} is non-degenerate; 
\item let $\{\xi_1, \ldots, \xi_{g\cdot [k:\Q]}\}$ and $\{\gamma_1, \ldots, \gamma_{g \cdot [k:\Q]}\}$ be bases of $H^2_{\mathscr{M}}(C, \Q(2))_{\Z}$ and  $H_1(C(\C), \Z)^-$ respectively, then 
$$\operatorname{det}(r_{\D}(\xi_i)(\gamma_j))_{i, j} \equiv L^{(g \cdot [k: \Q])}(C, 0)\equiv 
\frac{w N}{(2\pi)^{2g\cdot[k:\Q]}}L(C,2) \quad \pmod{\Q^*},$$
where $L^{(r)}(C, s)$ is the $r$-th derivative of $L(C, s)$, $w$ is the root number of $L(C,s)$, 
and $N$ is the norm {\rm(}to $\Q)$ of the conductor of $C$. 
\end{enumerate}
\end{conj}

\section{Elements in motivic cohomology} \label{elements}
Let $X_{N, n}$ be the generalized Fermat curve over $k_{X_{N, n}}:=\Q(\la_1, \ldots, \la_n)$ defined by
\begin{align*}
X_{N, n} : \left\{
\begin{array}{l}
X^N+(-1)^nY_1^N=\lambda_1 W^N \\
X^N+(-1)^nY_2^N=\lambda_2 W^N \\
\qquad \vdots \\
X^N+(-1)^nY_n^N=\lambda_n W^N
\end{array}
\right.
\subset \p^{n+1}(\C), 
\end{align*}
where $\la_1, \ldots, \la_n \in \bQ^\times$ are different from each other.   
Let $x=X/W$ and $y_i=Y_i/W$ ($i=1, \ldots, n$) be the affine coordinates. 
We recall the polynomial 
\begin{align} \label{charpoly}
\Phi_{X_{N, n}}(T)=T\Big(\prod_{i=1}^n(T-\lambda_i)\Big)-1 \in \overline{\Q}[T].
\end{align}
\begin{thm}\label{mot}
Suppose that $\Phi_{X_{N, n}}(T)$ is decomposed in $\overline{\Q}[T]$ as follows: 
\begin{align} \label{Phi}
\Phi_{X_{N, n}}(T)=(T-\nu_1)\cdots (T-\nu_n)(T-(-1)^n(\nu_1 \cdots \nu_n)^{-1}),
\end{align}
where $\nu_1, \ldots, \nu_n$ are roots of unity in $\overline{\Q}$. 
Then there is a positive integer $M$ such that $\{1-xy_1 \ldots y_n, x\}^M \in \operatorname{Ker}{\tau}$, i.e. 
$$\xi_{X_{N,n}}:=\{1-xy_1 \ldots y_n, x\} \in H^2_{\mathscr{M}}(X_{N, n}, \Q(2)).$$ 
\end{thm}

\begin{proof}
For any $\overline{k}$-point $P=[X :Y_1 : \cdots : Y_n : W] \in X_{N, n}(\overline{\Q})$, we show that $\tau_P(\xi_{X_{N,n}})$ is torsion. 
When $X=0$, then $\ord_P(1-xy_1 \cdots y_n)=0$, hence
$$\tau_P(\xi_{X_{N,n}})=(1-x(P)y_1(P) \cdots y_n(P))^{\ord_P(x)}=1. $$
We can similarly show that $\tau_P(\xi_{X_{N,n}})=1$ for $Y_i=0$. 
When $W=0$, we put $P=[1 : \zeta_{2N} \zeta_N^{i_1} : \cdots : \zeta_{2N} \zeta_{N}^{i_n} : 0]$, where $1 \leq i_1, \ldots, i_n \leq N$.  
Since $\ord_P(x)=-1$ and $\ord_P(1-xy_1 \cdots y_n)=-(n+1)$, 
the tame symbol at $P$ is 
$$\tau_P(\xi_{X_{N,n}})=-\frac{x^{n+1}}{1-xy_1 \ldots y_n}(P)=\frac{1}{\zeta_{2N}^{n}\zeta_N^{i_1+\cdots i_n}},$$
hence $\tau_P(\xi_{X_{N,n}})$ is torsion. 
Suppose that $XY_1 \cdots Y_n W \neq 0$. 
We write $P=[\alpha : \beta_1 : \cdots \beta_n :1]$ since $W \neq 0$. 
Then $\ord_P(x)=0$, hence we need to show that
$$\tau_P(\xi_{X_{N,n}})=\frac{1}{x^{\ord_P(1-xy_1 \cdots y_n)}}(P)=\frac{1}{\alpha^{\ord_P(1-xy_1 \cdots y_n)}}(P)$$
is torsion. 
When $\ord_P(1-xy_1 \cdots y_n)=0$, then we have $\tau_P=1$, so we consider the case $\ord_P(1-xy_1 \cdots y_n)>0$. 
Note that $\al$, $\be_1, \ldots, \be_n$ satisfy the following equations:
$$
 \left\{
\begin{array}{l}
\al\be_1\cdots \be_n=1 \\
\al^N+(-1)^n\be_1^N=\lambda_1  \\
\al^N+(-1)^n\be_2^N=\lambda_2 \\
\qquad \vdots \\
\al^N+(-1)^n\be_n^N=\lambda_n. 
\end{array}
\right.
$$
By eliminating $\beta_1, \ldots, \beta_n$ and putting $\alpha^N=T$, we have
\begin{align} \label{tame}
T\prod(T-\lambda_i)-1=0\quad  \Longleftrightarrow \quad \Phi_{X_{N, n}}(T)=0. 
\end{align}
By the assumption \eqref{Phi}, $\Phi_{X_{N, n}}(T)$ has a decomposition in $\overline{\Q}[T]$:
\begin{align*} 
\Phi_{X_{N, n}}(T)=(T-\nu_1)\cdots (T-\nu_n)(T-(-1)^n(\nu_1\cdots \nu_n)^{-1})=0,
\end{align*}
where $\nu_i$'s are roots of unity.
Then all roots of \eqref{tame} are roots of unity, which finishes the proof.   
\end{proof}

Let $E_{N, n}$ be the superelliptic curve over $k_{E_{N, n}}:=\Q(\lambda_1, \ldots, \lambda_n)^{S_n}$ which is a 
smooth projective model of  the smooth affine curve
$$Y^N=f(X)=X(X-\lambda_1) \cdots (X-\lambda_n). $$
Here, $S_n$ is the $n$-th symmetric group acting on $\Q(\la_1, \ldots, \la_n)$ by permuting $\la_1, \ldots, \la_n$. 
Let 
$$\pi \colon E_{N, n} \times_{k_{E_{N, n}}} k_{X_{N, n}} \to E_{N, n}$$
be a base change. 
Then we have a surjective map over $k_{X_{N, n}}$ which is locally given by 
$$p \colon X_{N, n} \to E_{N, n} \times_{k_{E_{N, n}}} k_{X_{N, n}} ; \quad (x,y_1, \ldots, y_n) \mapsto (X, Y)=(x^N, xy_1 \cdots y_n) . $$
We construct an element of the motivic cohomology of $E_{N, n}$. 
Since
 $$(p^* \circ \pi^{*})(\{1-Y, X\})=N \xi_{X_{N,n}} \in H_{\mathscr{M}}^2(X_{N, n}, \Q(2))$$
 and $p^*$, $\pi^*$ are injective by the norm argument, 
we define an element of the motivic cohomology of $E_{N, n}$ by
$$\{1-Y, X\} \in H_{\mathscr{M}}^2(E_{N, n}, \Q(2)). $$

Let $E_{\Q(\mu_N)}:=E_{N, n} \times_{k_{E_{N, n}}} \Q(\mu_N) (\la_1, \ldots, \la_n)^{S_n}$.  
Then $\mu_N$ acts on $E_{\Q(\mu_N)}$ by 
$$(X, Y) \mapsto (X, \zeta  Y) \quad (\zeta \in \mu_N). $$
For $\{1-Y, X\} \in H_{\mathscr{M}}^2(E_{\Q(\mu_N)}, \Q(2)) \subset K_2^M(\Q(\mu_N)(E_{N, n}))$, 
we construct further elements in the motivic cohomology using this action: for $\zeta \in \mu_N$, 
$$\zeta^*\{1-Y, X\}=\{1-\zeta Y, X\} \in H_{\mathscr{M}}^2(E_{\Q(\mu_N)}, \Q(2)) 
\subset K_2^M(\Q(\zeta_N)(E_{N, n})). $$

\begin{prop} \label{intE}
Keep the notation and assumption as above. 
Further, assume $f(X)\in \Z[X]$. 
The symbol $\{1-Y, X\}$ is integral, i.e.
$$\{1-Y, X\} \in H^2_{\mathscr{M}}(E_{N, n}, \Q(2))_{\Z}. $$
\end{prop}
\begin{proof}
We mimic the proof of \cite[p.359, Theorem 8.3]{DJZ}. 
For simplicity, put $C=E_{N,n}$. 
Since $f(X)\in \Z[X]$ is monic, we can find  an integral proper flat model $\mathcal{C}$ of $C$ 
over $\Z$. In fact, we have only to take the normalization of the morphism 
${\rm Spec}\hspace{0.3mm}\Z[X,Y]/(Y^N-f(X))\lra \mathbb{P}^1_\Z,\ (X,Y)\mapsto X$. 

Let $\widetilde{\mathcal{C}}$ be a regular proper model over $\Z$ of $C$ 
with a modification $\pi:\widetilde{\mathcal{C}}\lra \mathcal{C}$.

For any prime $p$, let $\mathcal{D}$ be any irreducible divisor of 
the special fiber $\widetilde{\mathcal{C}}_p$ of $\widetilde{\mathcal{C}}$ at $p$. 
We may prove the vanishing of the tame symbol $T_\mathcal{D}$ (see \cite[p.343, (3.5)]{DJZ} 
for the tame symbol).
There are two cases we need to consider:
\begin{enumerate}
\item $\pi(\mathcal{D})$ is an irreducible component of the 
special fiber $\mathcal{C}_p$ of $\mathcal{C}$ at $p$;
\item $\pi(\mathcal{D})$ is a singular point of $\mathcal{C}_p$. 
\end{enumerate}

For the first case, we may work on $\mathcal{C}_p$ and also on the dense open subscheme 
$\mathcal{C}^\circ_p:={\rm Spec}\hspace{0.3mm}\F_p[X,Y]/(Y^N-f(X))$. Put 
$\mathcal{D}':=\pi(\mathcal{D})\cap \mathcal{C}^\circ_p$ and we may assume that 
it is dense in an irreducible component of $\mathcal{C}^\circ_p$.  
If $Y\equiv 1$ on $\mathcal{D}'$, then $\mathcal{D}'$ has to be of dimension zero but 
it is impossible. Thus, $\ord_{\mathcal{D}}(Y-1)=\ord_{\mathcal{D}'}(Y-1)=0$ and thus $T_\mathcal{D}=1$. 

For the second case, 
let $\pi(\mathcal{D}):=(x_0,y_0)$ be the singular point of $\mathcal{C}^\circ_p$. Then 
$Ny^{N-1}_0=f'(x_0)=0$. If $p\nmid N$, then $y_0=0$. Thus, $1-Y\equiv 1-y_0=1$ on 
$\mathcal{D}$. Hence, $T_\mathcal{D}=1$. When $p \mid N$, if $x_0=0$, $y^N_0=f(x_0)=0$ 
so that $1-y_0=1$ and $\ord_{\mathcal{D}}(Y-1)=0$, otherwise $\ord_{\mathcal{D}}(X)=0$ since 
$x_0\neq 0$. 
Thus, we have  $T_\mathcal{D}=1$.

What remains is to check it for the points at infinite. Let $k$ be the least positive  
integer $k$ such that $kN-(n+1)\ge 0$. 
Substituting $(X,Y)=(V^{-1},WV^{-k})$, the equation $Y^N=f(X)$ transforms into 
$W^N=V^{Nk-(n+1)}\prod_{i=1}^n(1-\la_i V)$. 
Any infinite points map to $(V,W)=(0,0)$ if $Nk > (n+1)$, and $(V, W)=(0, \zeta)$ ($\zeta \in \mu_N$) if $Nk=n+1$.   
We may work on the local coordinates $(V,W)$.
We may observe 
$$\{1-WV^{-k},V^{-1}\}=\{WV^{-k},V^{-1}\}+\{V^kW^{-1}-1,V^{-1}\}.$$
It follows from $\{V^{-1},V\}=0$ that 
$$\{WV^{-k},V^{-1}\}=N^{-1}\{W^NV^{-kN},V^{-1}\}=-
N^{-1}\{V^{-(n+1)}\prod_{i=1}^n(1-\la_i V),V\}=
-N^{-1}\{\prod_{i=1}^n(1-\la_i V),V\}.$$
Since $1-\la_i V\equiv 1$ on $\mathcal{D}$ for each $i$ ($1\le i\le n$), we have  
$T_{\mathcal{D}}(\{WV^{-k},V^{-1}\})=1$. 
Notice that $\ord_{\mathcal{D}}(V^kW^{-1})=\frac{kN}{kN-(n+1)}>0$ if $kN > n+1$, and $\ord_{\mathcal{D}}(V^kW^{-1})=k>0$ if $kN = n+1$. Thus, 
$V^kW^{-1}-1\equiv -1$ on $\mathcal{D}$. Hence, $\ord_{\mathcal{D}}(V^kW^{-1}-1)=0$ 
and it yields $T_{\mathcal{D}}(\{V^kW^{-1}-1,V^{-1}\})=1$. 
Thus, we complete the proof.
\end{proof}

Let $\mathscr{X}_{N, n}$ (resp. $\mathscr{E}_{N, n}$) be a regular proper model of $X_{N, n}$ (resp. $E_{N, n}$) over $\mathscr{O}_{k_{X_{N, n}}}$ (resp. $\mathscr{O}_{k_{E_{N, n}}}$) and 
$\mathscr{X}_{N, n, v}$ (resp. $\mathscr{E}_{N, n, v}$) be the reduction at a prime $v$.  
Then we have the commutative diagram:
\[
  \begin{CD}
   H^2_{\mathscr{M}}(E_{N, n}, \Q(2))_{\Z} @>>>  H^2_{\mathscr{M}}(E_{N, n}, \Q(2))  @>{\partial}>> \displaystyle \coprod_v K'_1\mathscr{E}_{N, n, v} \otimes \Q, \\
   @V{p^* \circ \pi^*}VV   @V{p^* \circ \pi^*}VV   @V{p^* \circ \pi^*}VV \\
  H^2_{\mathscr{M}}(X_{N, n}, \Q(2))_{\Z} @>>>  H^2_{\mathscr{M}}(X_{N, n}, \Q(2))  @>{\partial}>> \displaystyle \coprod_v K'_1\mathscr{X}_{N, n, v} \otimes \Q 
  \end{CD}
\]
where $\partial$ is the boundary map in $K$-theory (cf. \cite[p.226, (4.7.1)]{Ram}) which is nothing but the map 
\cite[(3.5)]{DJZ}.  
Thus as an immediate consequence from this diagram with Proposition \ref{intE}, we have 
the following result:
\begin{cor}\label{intX} Keep the notation and assumption as in Proposition \ref{intE}. 
Then, the symbol $\xi_{X_{N,n}}$ is integral, i.e.
$$\xi_{X_{N,n}} \in H^2_{\mathscr{M}}(X_{N, n}, \Q(2))_{\Z}.$$
\end{cor}

\section{Computation of regulators}  \label{computation}
Let $\Phi_{X_{N, n}}(T)$ be the polynomial defined in \eqref{charpoly} and 
assume that 
$$f(T)=\Phi_{X_{N, n}}(T)+1=T(T-\la_1) \cdots (T-\la_n)$$
satisfies (\ref{thmassump}).
We may assume that $\la_n \in (0, 1)$ without loss of generality. 
\begin{dfn} \label{cycle} 
Define a path $\gamma$ on $E_{N, n}(\C)$ by
 $$\gamma \colon [0, 1] \to E_{N, n}(\C); \quad  t \mapsto \left(t, \sqrt[N]{f(t)}\right), $$
 where we take the branch of $\sqrt[N]{f(t)}$ along $\R_{\leq 0}$. 
 Then $\delta =\gamma -\overline{\gamma}$ is a cycle and defines an element of $H_1(E_{N, n}(\C), \Z)$. 
\end{dfn}
 
By the assumption (\ref{thmassump})-(iv), $f(t)\leq 0$ for $t \in [0, \lambda_n]$ and $f(t) \geq 0$ for $[\lambda_n, 1]$. 
If $t \in (0, \lambda_n)$, then we have $\operatorname{arg}(f(t))=\pi$ and 
$$\sqrt[N]{f(t)}=\exp{\left(\frac1N\log(-f(t))\right)} \zeta_{2N}=\sqrt[N]{-f(t)}\zeta_{2N}.$$
Thus, we have $$\delta =\gamma -\overline{\gamma} \in H_1(E_{N, n}(\C), \Z)^{-}.$$ 

\begin{thm} \label{main:1}
Under (\ref{thmassump}), it holds that  
\begin{align*}
r_{\D}(\{1-Y, X\})(\delta) = &-\frac{1}{2\pi} \notag
\sum_{k=1}^{\infty}
\frac1k
 ((-1)^{n-1}\la_1 \cdots \la_n)^{\frac{k}{N}}\la_n^{\frac{k}{N}}
\sin\frac{\pi k}{N}B\left(\frac{k}{N}, \frac{k}{N} \right) \\
& \times F^{(n-1)}_D \left. \left( 
\frac{k}{N}, -\frac{k}{N}, \ldots, -\frac{k}{N}, \frac{2k}{N}+1
\ \right| 
\frac{\la_n}{\la_1}, \ldots, \frac{\la_{n}}{\la_{n-1}}
\right), 
\end{align*}
and the right hand side is non-zero. 
In particular, $\{1-X,Y\}\in H^2_{\mathscr{M}}(E_{N, n}, \Q(2))$ and 
$\xi_{X_{N,n}}\in H^2_{\mathscr{M}}(X_{N, n}, \Q(2))$ are both non-trivial. 
\end{thm}

\begin{proof}
We treat the case when $n$ is even since the other case is similarly handled. 
For sufficiently small $\epsilon>0$, let $\delta_{\epsilon}$ be a cycle such that $\delta_{\epsilon}$ belongs to the homology class of $\delta$ and avoids the zeros and poles of $1-Y$ and $X$. 
We apply the regulator formula (Proposition \ref{regformula}) with the base point $(1-\epsilon, \sqrt[N]{f(1-\epsilon)})$, we have
$$r_{\D}(\{1-Y, X\})(\delta _{\epsilon})=\frac{1}{2\pi}\operatorname{Im} \left( \int_{\delta_{\epsilon}} \log(1-Y)d \log X -\log(1-\epsilon) \int_{\delta_{\epsilon}} d \log(1-Y)\right).  $$
The second term converges to $0$ as $\epsilon$ tends to $0$. Since the first term does not depend on the choice of the base point (depends only on the choice of a homology class), we have
\begin{align*}
r_{\D}(\{1-Y, X\})(\delta)&=\frac{1}{2\pi}\operatorname{Im}  \int_{\delta} \log(1-Y)d \log X \\
&=\frac{1}{2\pi} \operatorname{Im} \left(\int_{\gamma}-\int_{\overline{\gamma}}\right) \log(1-Y)d \log X \\
&=\frac{1}{2\pi} \operatorname{Im} (I_1- I_2), 
\end{align*}
where 
\begin{align*}
&I_1= \int_{\gamma} \log(1-Y)d \log X , \quad I_2= \int_{\overline{\gamma}}  \log(1-Y)d \log X.
\end{align*}
Then 
\begin{align*}
I_1&=\int_0^1\log(1-\sqrt[N]{f(t)}) d \log t \\
&=\int_0^{\lambda_n}\log(1-\zeta_{2N}\sqrt[N]{-f(t)}) \frac{dt}{t} + \int_{\lambda_n}^1\log(1-\sqrt[N]{f(t)}) \frac{dt}{t}.
\end{align*}
By the assumption (iii), considering the Taylor expansion of $\log(1-Y)$, we have 
\begin{align*}
I_1&=-\sum_{k=1}^{\infty}
\frac1k
\left\{
\zeta_{2N}^k
\int_0^{\lambda_n}
t^{\frac{k}{N}-1}
(\lambda_1-t)^{\frac{k}{N}}
\cdots
(\lambda_{n-1}-t)^{\frac{k}{N}}
(t-\lambda_n)^{\frac{k}{N}}
dt \right. \\
& \qquad \qquad  \qquad \qquad   \qquad \qquad  \qquad \qquad \left. 
+
\int_{\lambda_n}^1
t^{\frac{k}{N}-1}
(\lambda_1-t)^{\frac{k}{N}}
\cdots
(\lambda_n-t)^{\frac{k}{N}}
dt
\right\}
\end{align*}
Same as above, we have
\begin{align*}
I_2&=\int_0^1\log(1-\overline{\sqrt[N]{f(t)}}) d \log t \\
&=\int_0^{\lambda_n}\log(1-\zeta^{-1}_{2N}\sqrt[N]{-f(t)}) \frac{dt}{t} + \int_{\lambda_n}^1\log(1-\sqrt[N]{f(t)}) \frac{dt}{t},
\end{align*}
hence we have
\begin{align*}
I_2&=-\sum_{k=1}^{\infty}
\frac1k
\left\{
\zeta_{2N}^{-k}
\int_0^{\lambda_n}
t^{\frac{k}{N}-1}
(\lambda_1-t)^{\frac{k}{N}}
\cdots
(\lambda_{n-1}-t)^{\frac{k}{N}}
(t-\lambda_n)^{\frac{k}{N}}
dt  \right. \\
& \qquad \qquad  \qquad \qquad   \qquad \qquad  \qquad \qquad \left. 
+
\int_{\lambda_n}^1
t^{\frac{k}{N}-1}
(\lambda_1-t)^{\frac{k}{N}}
\cdots
(\lambda_n-t)^{\frac{k}{N}}
dt
\right\}.
\end{align*}
Therefore, we have
\begin{align*}
I_1-I_2&=
\sum_{k=1}^{\infty}
\frac1k
(\zeta_{2N}^{-k}-\zeta_{2N}^k)
\int_0^{\lambda_n}
t^{\frac{k}{N}-1}
(\lambda_1-t)^{\frac{k}{N}}
\cdots
(\la_{n-1}-t)^{\frac{k}{N}}
(t-\lambda_n)^{\frac{k}{N}}
dt \\
&=-2 \sqrt{-1}
\sum_{k=1}^{\infty}
\frac1k
\sin\frac{\pi k}{N}
\int_0^{\lambda_n}
t^{\frac{k}{N}-1}
(\lambda_1-t)^{\frac{k}{N}}
\cdots
(\la_{n-1}-t)^{\frac{k}{N}}
(t-\lambda_n)^{\frac{k}{N}}
dt. 
\end{align*}
By substituting $s=t/\lambda_n$, we have 
\begin{align} \label{reg1}
&r_{\D}(\{1-Y, X\})(\delta) \\
&=-\frac{1}{\pi} \notag
\sum_{k=1}^{\infty}
\frac1k
\la_n^{\frac{k}{N}}
\sin\frac{\pi k}{N}
\int_0^{1}
s^{\frac{k}{N}-1}
(\lambda_1-\la_ns)^{\frac{k}{N}}
\cdots
(\lambda_{n-1}-\la_ns)^{\frac{k}{N}}
(\la_ns-\la_n)^{\frac{k}{N}}
ds \\
&=-\frac{1}{\pi} \notag
\sum_{k=1}^{\infty}
\frac1k
 (-\la_1 \cdots \la_n)^{\frac{k}N} \la_n^{\frac{k}{N}}
\sin\frac{\pi k}{N}
\int_0^{1}
s^{\frac{k}{N}-1}
(1-s)^{\frac{k}{N}}
\prod_{i=1}^{n-1}
\left(1-\frac{\la_n}{\la_i}s\right)^{\frac{k}{N}}
ds.
\end{align} 
By \eqref{int}, we have
\begin{align*} 
r_{\D}(\{1-Y, X\})(\delta) =&-\frac{1}{2\pi} \notag
\sum_{k=1}^{\infty}
\frac1k
 (-\la_1 \cdots \la_n)^{\frac{k}N} \la_n^{\frac{k}{N}}
\sin\frac{\pi k}{N}
B\left(\frac{k}{N}, \frac{k}{N} \right) \\
& \times 
F^{(n-1)}_D \left. \left( 
\frac{k}{N}, -\frac{k}{N}, \ldots, -\frac{k}{N}, \frac{2k}{N}+1
\ \right| 
\frac{\la_n}{\la_1}, \ldots, \frac{\la_{n}}{\la_{n-1}}
\right).
\end{align*}
Here, we used the relation $B(x, x+1)=1/2 B(x, x)$. 
By the assumption (ii)-(v), the right-hand side of \eqref{reg1} is positive except for $\sin\frac{\pi k}{N}$. 
Put 
$$\beta_K=\sum_{k \equiv K \pmod{2N}} \frac1k
(-\la_1 \cdots \la_n)^{\frac{k}N}
 \la_n^{\frac{k}{N}}
\int_0^{1}
s^{\frac{k}{N}-1}
(1-s)^{\frac{k}{N}}
\prod
\left(1-\frac{\la_n}{\la_i}s\right)^{\frac{k}{N}}
ds
\quad 
(K=1, \ldots, 2N).
$$
Since $\beta_1 > \cdots > \beta_{2N}>0$ by the assumption (iii), we have
$$r_{\D}(\{1-Y, X\})(\delta)=-\frac{1}{\pi} \sum_{K=1}^{N-1} \sin \frac{\pi K}{N} (\beta_K -\beta_{N+K})<0.$$
We conclude that $r_{\D}(\{1-Y, X\})(\delta) \neq 0$.
\end{proof}

We give two explicit (infinitely many) examples for Theorem \ref{main:1}. We also use it to numerically test 
the Beilinson conjecture for some cases among these in Section \ref{numerical_Beilinson}.

\begin{exa} \label{ex1}
Let 
\begin{align*} 
\Phi_{X_{N, n}}(T)&=(T^n-1)(T+1)=T^{n+1}+T^n-T-1. 
\end{align*}
Then $\Phi_{X_{N, n}}(T)$ and $f(T)$ satisfy (\ref{thmassump}), and the superelliptic curve $E_{N, n}$ is given by 
$$E_{N, n} : Y^N=X^{n+1} + X^n -X. $$
Since $(-1)^{n-1}\la_1 \cdots \la_n=1$, we have 
\begin{align*}
r_{\D}(\{1-Y, X\})(\delta) =&-\frac{1}{2\pi} \notag
\sum_{k=1}^{\infty}
\frac1k
 \la_n^{\frac{k}{N}}
\sin\frac{\pi k}{N}B\left(\frac{k}{N}, \frac{k}{N} \right) \\
& \times F^{(n-1)}_D \left. \left( 
\frac{k}{N}, -\frac{k}{N}, \ldots, -\frac{k}{N}, \frac{2k}{N}+1
\ \right| 
\frac{\la_n}{\la_1}, \ldots, \frac{\la_{n}}{\la_{n-1}}
\right).
\end{align*}
\end{exa}

\begin{exa} \label{ex2}
Let $\Phi_{X_{N, n}}(T)=(T+1)^{n+1-l}(T-1)^l$, 
where $l$ is an odd integer such that 
$$0<l < \frac{n+1}2, \quad \text{and} \quad \left|f\left(1- \frac{2l}{n+1}\right)\right| < 1. $$ 
Then $\Phi_{X_{N, n}}(T)$ and $f(T)$ satisfy (\ref{thmassump}), and the superelliptic curve $E_{N, n}$ is given by 
$$E_{N, n} : Y^N=(X+1)^{n+1-l} (X-1)^l +1. $$
Since $(-1)^{n-1}\la_1 \cdots \la_n=(-1)^{l-1} (n+1-2l)$, we have 
\begin{align*}
r_{\D}(\{1-Y, X\})(\delta) =&-\frac{1}{2\pi} \notag
\sum_{k=1}^{\infty}
\frac1k
((-1)^{l-1}(n+1-2l))^{\frac{k}N}
 \la_n^{\frac{k}{N}}
\sin\frac{\pi k}{N}B\left(\frac{k}{N}, \frac{k}{N} \right) \\
& \times F^{(n-1)}_D \left. \left( 
\frac{k}{N}, -\frac{k}{N}, \ldots, -\frac{k}{N}, \frac{2k}{N}+1
\ \right| 
\frac{\la_n}{\la_1}, \ldots, \frac{\la_{n}}{\la_{n-1}}
\right).
\end{align*}
\end{exa}

\section{Numerical verification of the Beilinson conjecture} \label{numerical_Beilinson}
In this section, we numerically test the Beilinson conjecture for $E_{N,n}$ with small genus. 
Let $g_{E_{N, n}}$ be the genus of the superelliptic curve $E_{N, n}$. 
By the Riemann-Hurwitz formula, we see that  
$$g_{E_{N, n}}=\frac{(N-1)n-\{\gcd(N, n+1)-1\}}{2}.$$
Thus, $E_{N, n}$ is an elliptic curve if and only if $(N, n)=(2, 2), (3, 2)$ and $(2, 3)$.  
We first handle the case of elliptic curves over $\Q$. 
For other cases, we give additional integral elements to test the Beilinson conjecture. 
Although this is away from our main theme, but it has an independent interests.   

\subsection{The case of elliptic curves over $\Q$} 
For any elliptic curve $E$ over $\Q$, we remark that the $L$-function $L(E,s)$ extends holomorphically in $s$ to the whole $\C$ by automorphy of $E$. Thus, we can consider 
$L^{(r)}(E,0)$ for any non-negative integer $r$. The same is true for the base change of $E$ to 
the cyclic extension over $\Q$ (\cite{Langlands}). 
We use this fact without further comment. 

For $n=2$, Tokiwa \cite{Tokiwa} studied the elliptic curves $E_2$ and $E_3$ over $\Q$ defined by 
$$
E_2 : Y^2=X^3+X^2-X,  \label{numericalex1} \ E_3 : Y^3=X^3+X^2-X, \label{numericalex2}
$$
which are quotients of $X_{2, 2}(\la_1, \la_2)$ and $X_{3, 2}(\la_1, \la_2)$ for $(\la_1, \la_2)=(-\frac{1+\sqrt{5}}2, -\frac{1-\sqrt{5}}2)$, respectively. 
In these cases, the regulators in Theorem \ref{main:intro} are described by using Gaussian hypergeometric functions, and he verified numerically the Beilinson conjecture \cite[Conjecture 6.13]{Tokiwa}:  
$$
r_{\D}(\{1-Y, X\}^2)(\delta) \approx 20 \cdot L^{(1)}(E_{2}, 0),\ 
r_{\D}(\{1-Y, X\}^2)(\delta) \approx \frac13 \cdot L^{(1)}(E_{3}, 0),
$$
respectively. 

We consider the remaining case $(N, n)=(2, 3)$. There are only two cases:
\begin{align*}
&E_{2, 3} : Y^2=(X^3-1)(X+1)+1=X^4+X^3-X, \\
&E'_{2, 3} : Y^2=(X+1)^3(X-1)+1=X^4+2X^3-2X. 
\end{align*}

We first consider $E_{2, 3}$ which is the elliptic curve in Example \ref{ex1} for $N=2$ and $n=3$. 
Put 
\begin{align*}
R: &
=
-\frac{1}{2\pi}\sum_{k=1}^{\infty}
\frac1k
 \la_3^{\frac{k}{2}}
\sin\frac{\pi k}{2}
B\left(\frac{k}{2}, \frac{k}{2} \right)
F_1 \left( \left. 
\frac{k}2, -\frac{k}2, -\frac{k}2, k+1
\ \right| 
\frac{\la_3}{\la_1}, \frac{\la_{3}}{\la_{2}}
\right),
\end{align*}
where $F_1(\ |\ x_1, x_2)=F_D^{(2)}(\ |\ x_1, x_2)$ is the Appell function. 
By using Mathematica, we have
$$R=-0.47095904334493274691418567400102924042656389674994 \ldots.$$
On the other hand, 
the Weierstrass model of $E_{2, 3}$ is 
$$Y^2=X^3-X+1, $$
whose conductor is $92$. 
By using Magma, we have 
$$L^{(1)}(E_{2, 3}, 0)=-1.4128771300347982407425570220030877212796916902498\ldots. $$
Thus, we have numerically 
$$R \approx \frac13 \cdot L^{(1)}(E_{2, 3}, 0)$$
with a precision of $50$ decimal digits. 

Next we consider $E'_{2, 3}$ 
which is the elliptic curve in Example \ref{ex2} for $N=2$, $n=3$ and $l=1$. 
Put 
\begin{align*}
R: &
=
-\frac{1}{2\pi}\sum_{k=1}^{\infty}
\frac1k
(2 \la_3)^{\frac{k}{2}}
\sin\frac{\pi k}{2}
B\left(\frac{k}{2}, \frac{k}{2} \right)
F_1 \left( \left. 
\frac{k}2, -\frac{k}2, -\frac{k}2, k+1
\ \right| 
\frac{\la_3}{\la_1}, \frac{\la_{3}}{\la_{2}}
\right). 
\end{align*}
By using Mathematica, we have
$$R=-0.66881964039649037504127102520475989774413452667566 \ldots. $$
On the other hand, the Wierstrass model of $E_{2, 3}$ is 
$$Y^2=X^3-4X+4, $$
whose conductor is $88$.  
By using Magma, we have 
$$L^{(1)}(E_{2, 3}, 0)=-1.3376392807929807500825420504095197954882690533513 \ldots. $$
Thus, we have numerically 
$$R \approx \frac12 \cdot L^{(1)}(E'_{2, 3}, 0)$$
with a precision of $50$ decimal digits. 

\subsection{The case of a hyperelliptic curve of genus $2$ over $\Q$} 
In this subsection, we consider the hyperelliptic curve of genus $2$ defined by 
\begin{align} \label{hyperelliptic}
E_{2,4}:Y^2=f(X)=(X^4-1)(X+1)+1=X^5+X^4-X. 
\end{align}
To check the Beilinson conjecture of \eqref{hyperelliptic} numerically, we need to construct one more integral element in the motivic cohomology. 

\begin{prop} \label{newK2element}
Let $E_{2, 2N}$ be the hyperelliptic curve of genus $N$ defined by 
\begin{align}
Y^2=X^{2N+1}+X^{2N}-X. \label{hyperelliptic2}
\end{align}
Then the symbol $\{Y-X^N, X\}$ defines an element of $H^2_{\mathscr{M}}(E_{2, 2N}, \Q(2))_\Z$. 
\end{prop}

\begin{proof}
 For any $\overline{\Q}$-point $P \in E_{2, 2N}(\overline{\Q})$, we show that the image under the tame symbol $\tau_P(\{Y-X^N, X\})$ is torsion. 
 Note that  
\begin{align*}
& \operatorname{div} (Y-X^N)= (0, 0) + \sum_{\zeta \in \mu_{2N}} (\zeta, \zeta^N) -(2N+1) (\infty), \\
& \operatorname{div} (X)= 2(0, 0) - 2 (\infty). 
\end{align*} 
If $P \neq (0, 0), (\zeta, \zeta^N), \infty$, then $\tau_P(\{Y-X^N, X\})=1$, hence it suffeces to show the case $P = (0, 0), (\zeta, \zeta^N), \infty$. 
If $P=(\zeta, \zeta^N)$, then $\tau_P(\{Y-X^N, X\})$ is torsion. 
If $P=(0, 0)$, then we have 
\begin{align*}
\tau_P(\{Y-X^N, X\})
&=- \left.\dfrac{(Y-X^N)^2}{X} \right|_{(X, Y)=(0, 0)} \\
&=-\left.\left(\frac{Y^N}{X} -2X^{N-1} Y +X^{2N-1} \right) \right|_{(X, Y)=(0, 0)} \\
&= - \left.\left(X^{2N}+X^{2N-1} -1 -2X^{N-1} Y +X^{2N-1}\right) \right|_{(X, Y)=(0, 0)} \\
&= 1, 
\end{align*}
hence $\tau_P(\{Y-X^N, X\})$ is trivial.  
Substituting $(X, Y)=(V^{-1}, WV^{-{(N+1)}})$, the equation \eqref{hyperelliptic2} transfoms into 
$$W^2=V+V^2-V^{2N-1}. $$
Then the infinite points map to $(V, W)=(0, 0)$ and we may work on the local coordinates $(V, W)$. 
We have 
\begin{align*}
\tau_P(\{Y-X^N, X\}) 
&=- \left. \dfrac{X^{2N+1}}{(Y-X^N)^2}\right|_{(X, Y)=P} \\
&= - \left.\dfrac{V^{-(2N+1)}}{(WV^{-(N+1)} - V^{-N})^2} \right|_{(V, W)=(0, 0)}\\
&=-\left. \dfrac1{W^2V^{-1} -2W+V}\right|_{(V, W)=(0, 0)} \\
&=- \left.\dfrac1{1+V-V^{2N-2}-2W+V} \right|_{(V, W)=(0, 0)} \\
&=- 1, 
\end{align*}
hence $\tau_P(\{Y-X^N, X\})$ is torsion. 

We can also check the integrality of our element as in Proposition \ref{intE} and 
thus we omit the details. 
\end{proof}

Let us turn to consider $E_{2,4}$. 
By Proposition \ref{newK2element},  
we have $\{Y-X^2, X\} \in H^2_{\mathscr{M}}(E_{2, 4}, \Q(2))_{\Z}$.  
Let $\delta \in H_1(E_{2, 4}(\C), \Z)^-$ as before. 
We construct another cycle $\delta'$ as follows. 
Define a path  $\gamma'$ on $E_{2, 4}(\C)$ by 
  $$\gamma' \colon ( -\infty, -1] \to E_{2, 4}(\C); \quad  t \mapsto \left(t, \sqrt[N]{f(t)}\right), $$
 where the branches are taken in $\R_{\leq 0}$.   
Let $\la_1$ be a unique negative real roots of $f(X)=0$. Then 
we have $f(t)\leq 0$ for $t \in (-\infty, \la_1]$ and $f(t)\ge  0$ for $t \in [\la_1, -1]$.  
Similarly the construction of $\delta$, $\delta' =\gamma' -\overline{\gamma'}$ becomes a cycle and defines an element of $H_1(E_{2, 4}(\C), \Z)^-$ (see Definition \ref{cycle}). 

We define the regulator determinant $R$ by 
$$R=\left|
\operatorname{det}
\begin{pmatrix}
r_{\D}(\{1-Y, X\})(\delta) &  r_{\D}(\{1-Y, X\})(\delta') \\
r_{\D}(\{Y-X^2, X\})(\delta)  & r_{\D}(\{Y-X^2, X\}(\delta')  \\
\end{pmatrix}
\right|.$$
By using Mathematica, we have 
\begin{align*}
&r_{\D}(\{1-Y, X\})(\delta)=-0.50222774458567807234309977656150357407699614640711 \ldots,   \\
&r_{\D}(\{1-Y, X\})(\delta')= -0.25869891405241850202655302706581105434092244125967 \ldots , \\
&r_{\D}(\{Y-X^2, X\})(\delta)=0.28090959342921653089311129619302051571945475691571 \ldots,   \\
&r_{\D}(\{Y-X^2, X\}(\delta')=-0.74910711218020537922324806983278156297750025766609 \ldots, 
\end{align*}
hence we obtain 
$$R=0.44889338217039979100950815236832711055130985508762 \ldots. $$
Since $E_{2,4}$ is of genus 2, by \cite[p.156, Theorem 1.1.2]{BCGP} it is potentially automorphic and $L(E_{2,4},s)$ has 
the expected functional equation. Thus, unconditionally, we have 
$L^{(2)}(E_{2, 4}, 0)=\frac{w\cdot N}{(2\pi)^4}L(C,2)$ with the root number 
$w=-1$ and the conductor $N=4528=2^4\cdot 283$ (see
\cite[Genus 2 curve 4528.b.72448.1]{LMFDB} for $w$ and $N$). Hence, 
we can consider the value $L^{(2)}(E_{2, 4}, 0)$, and 
by using Magma, we have 
$$L^{(2)}(E_{2, 4}, 0)=-3.5911470573631983280760652189466168844104788407010 \ldots. $$
Thus, we have numerically 
$$R  \approx - \frac18 \cdot L^{(2)}(E_{2, 4}, 0)$$
with a precision of $50$ decimal digits.

\subsection{The case of a superelliptic curve of genus $3$ over $\Q$} 
Let $E_{4,2}$ be the superelliptic curve of genus $3$ over $\Q$ defined by 
$$Y^4=X^3+X^2-X, $$
which is a quotient of $X_{4, 2}(\la_1, \la_2)$ for $(\la_1, \la_2)=(-\frac{1+\sqrt{5}}2, -\frac{1-\sqrt{5}}2)$. 
We have the automorphism $\sigma$ of order $4$ defined by 
$$\sigma \colon E_{4,2} \to E_{4,2}; \quad (X, Y) \mapsto \left(-\frac1X, \frac{Y}{X^3} \right). $$
Let $\xi:=\{1-Y, X\} \in H^2_{\mathscr{M}}(E_{4,2}, \Q(2))_{\Z}$ and $\delta \in H_1(E_{4,2}(\C), \Z)^-$ as before. 
By using the automorphism $\sigma$, we construct $3$ elements in the integral part of the motivic cohomology 
$\xi$, $\sigma^* \xi$ and $(\sigma^2)^* \xi$, and 
$3$ homology cycles $\delta$, $\sigma_*\delta$ and $(\sigma^2)_*\delta$.  
We define the regulator determinant $R$ by 
$$R=\left|
\operatorname{det}
\begin{pmatrix}
r_{\D}(\xi)(\delta) & r_{\D}(\xi)(\sigma_*\delta) & r_{\D}(\xi)((\sigma^2)_*\delta) \\
r_{\D}(\sigma^*\xi)(\delta) & r_{\D}(\sigma^*\xi)(\sigma_*\delta)  & r_{\D}(\sigma^*\xi)((\sigma^2)_*\delta)\\
r_{\D}((\sigma^2)^*\xi)(\delta) & r_{\D}((\sigma^2)^*\xi)(\sigma_*\delta) & r_{\D}((\sigma^2)^*\xi)((\sigma^2)_*\delta) 
\end{pmatrix}
\right|.$$
By using Mathematica, we have 
\begin{align*}
r_{\D}(\xi)(\delta)&=r_{\D}((\sigma^2)^*\xi)((\sigma^2)_*\delta) \\
& =-0.98844708489657058704834105512085052973343416105202 \ldots , \\
r_{\D}(\xi)(\sigma_*\delta)&=r_{\D}(\sigma^*\xi)(\delta) \\
&=-0.70555740813628736374241199707594442755919668795397 \ldots , \\
r_{\D}(\xi)((\sigma^2)_*\delta) &= r_{\D}(\sigma^*\xi)(\sigma_*\delta)=r_{\D}((\sigma^2)^*\xi)(\delta) \\
& =0.58887994509588812937108115517011098298739696261455 \ldots , \\
r_{\D}(\sigma^*\xi)((\sigma^2)_*\delta) &= r_{\D}((\sigma^2)^*\xi)(\sigma_*\delta) \\
&=0.30599026833560490606515209712520488081315948951649 \ldots, 
\end{align*}
hence we obtain 
$$R=0.70147792522235455249515324047670619687499664105676 \ldots. $$
On the other hand, we have an isogeny over $\Q$ 
$$\jac(E_{4,2}) \stackrel{\Q}{\sim} E_2 \times \operatorname{Res}_{K/\Q} E', $$
where $E'$ is the elliptic curve over $K=\Q(i)$ defined by 
$$Y^2=X^3+(-8 i +4)X$$
and it is the quotient curve $E_{4,2}/\langle \rho \rangle $ of $E_{4,2}$ for 
$\rho \colon (X,Y)\mapsto (-X^{-1},iYX^{-1})$. In fact, the quotient map $E_{4,2}\lra 
E'=E_{4,2}/\langle \rho \rangle$ is given by 
$(X,Y)\mapsto ((X+i)^2Y^{-1},(X+i)(X^2+(2-2i)X-1)Y^{-3})$. 
Therefore, we have 
$$L(E_{4,2}, s)=L(E_2, s)L(E', s). $$
Note that both of $L(E_2, s)$ and $L(E', s)$ have holomorphic continuation to $\C$ in $s$.  
Since $L(E_2, 0)=L(E', 0)=L^{(1)}(E', 0)=0$, we compute by using Magma
\begin{align*}
L^{(3)}(E_{4,2}, 0)&=6\cdot L^{(0)}(E_2, 0) L^{(2)}(E', 0) \\
&=0.6475180848206349715339876065938826432692415861783 \ldots. 
\end{align*}
Thus, we have numerically 
$$R \approx \dfrac{13}{12} \cdot L^{(3)}(E_{4,2}, 2)$$
with a precision of $50$ decimal digits.

\subsection{Beilinson conjecture over the quadratic field}\label{Beilinson over K}
We consider the elliptic curve $E_{3, 2}$ over $\Q$ in Example \ref{ex1}.  
Let $K=\Q(\zeta_3)$ and $E_K=E_{3, 2} \times_{\operatorname{Spec} \Q} \spec K$ be the base change. Since $E_K$ is automorphic, $L(E_K,s)$ has holomorphic continuation to $\C$ in $s$. 
We consider the Beilinson conjecture for $E_K$. 
Note that $H_1(E_K(\C), \Z)^-$ is the $2$-dimensional; we have an isomorphism 
$$H_1(E_K(\C), \Q)^- \cong \left(H_1(E_{3, 2}(\C), \Z) \oplus H_1(E_{3, 2}(\C), \Z) \right)^- \cong H_1(E_{3, 2}(\C), \Z),$$
where the complex conjugation acts on this by swapping conjugate pairs of complex
embeddings of $K$.

Let $\xi \in H^2_{\mathscr{M}}(E_K, \Q(2))_{\Z}$ and $\delta \in H_1(E_K(\C), \Z)^-$ as before. 
For two elements $\xi, (\zeta-\overline{\zeta})^* \xi \in H_{\mathscr{M}}^2(E_K, \Q(2))_{\Z}$, and two cycles $\delta,  (\zeta-\overline{\zeta})_* \delta\in H_1(E_K(\C), \Z)^-$ ($\zeta \in \mu_3 \setminus \{1\}$), we define the regulator determinant $R$ by 
$$R=\left|
\operatorname{det}
\begin{pmatrix}
r_{\D}(\xi)(\delta) &r_{\D}((\zeta-\overline{\zeta})^* \xi)(\delta) \\
 r_{\D}(\xi)((\zeta-\overline{\zeta})_* \delta) & r_{\D}((\zeta-\overline{\zeta})^* \xi)((\zeta-\overline{\zeta})_* \delta) \\
\end{pmatrix}
\right|.$$
Since $(\zeta-\overline{\zeta})_* \delta \in H_1(E_{3,2}(\C), \Q)^+$, where $+$ denotes the part fixed by $F_{\infty} \otimes c$, we have  $r_{\D}(\xi)((\zeta-\overline{\zeta})_* \delta)=0$. 
Therefore, we have 
$$R=\left| r_{\D}(\xi)(\delta) \cdot  r_{\D}((\zeta-\overline{\zeta})^* \xi)((\zeta-\overline{\zeta})_* \delta) \right|. $$
By using Mathematica, we have 
\begin{align*}
r_{\D}(\xi)(\delta)&= r_{\D}((\zeta-\overline{\zeta})^* \xi)((\zeta-\overline{\zeta})_* \delta) \\
& =-0.73225693138562217781835762443363615482360800039157 \ldots , 
\end{align*}
hence we obtain 
$$R=0.53620021356228778605831812308182307487055561499420 \ldots. $$

On the other hand, by using Magma, we have 
$$L^{(2)}(E_K, 0)=38.606415376484720596198904861891261390680004279582 \ldots. $$
Thus, we have numerically 
$$R \approx \frac{1}{72} \cdot L^{(2)}(E_K, 0)$$  
with a precision of $50$ decimal digits.

\end{document}